\documentclass[12pt]{amsart}
\usepackage{mathrsfs}
\usepackage{amssymb}
\usepackage{verbatim}
\usepackage{hyperref}
\usepackage{color}
\usepackage [latin1]{inputenc}

\usepackage{enumerate}

\begin{document}

\newtheorem*{thmC}{Theorem C}
\newtheorem*{thmA}{Theorem A}
\newtheorem*{thmB}{Theorem B}
\newtheorem{theorem}{Theorem}    
\newtheorem{proposition}[theorem]{Proposition}
\newtheorem{conjecture}[theorem]{Conjecture}
\def\theconjecture{\unskip}
\newtheorem{corollary}[theorem]{Corollary}
\newtheorem{lemma}[theorem]{Lemma}
\newtheorem{assumption}[theorem]{Assumption}
\newtheorem{sublemma}[theorem]{Sublemma}
\newtheorem{observation}[theorem]{Observation}
\theoremstyle{definition}
\newtheorem{definition}{Definition}
\newtheorem{notation}[definition]{Notation}
\newtheorem{remark}[definition]{Remark}
\newtheorem{question}[definition]{Question}
\newtheorem{questions}[definition]{Questions}
\newtheorem{example}[definition]{Example}
\newtheorem{problem}[definition]{Problem}
\newtheorem{exercise}[definition]{Exercise}

\numberwithin{theorem}{section}
\numberwithin{definition}{section}
\numberwithin{equation}{section}

\def\earrow{{\mathbf e}}
\def\rarrow{{\mathbf r}}
\def\uarrow{{\mathbf u}}
\def\varrow{{\mathbf V}}
\def\tpar{T_{\rm par}}
\def\apar{A_{\rm par}}

\def\reals{{\mathbb R}}
\def\torus{{\mathbb T}}
\def\heis{{\mathbb H}}
\def\integers{{\mathbb Z}}
\def\naturals{{\mathbb N}}
\def\complex{{\mathbb C}\/}
\def\distance{\operatorname{distance}\,}
\def\support{\operatorname{support}\,}
\def\dist{\operatorname{dist}\,}
\def\Span{\operatorname{span}\,}
\def\degree{\operatorname{degree}\,}
\def\kernel{\operatorname{kernel}\,}
\def\dim{\operatorname{dim}\,}
\def\codim{\operatorname{codim}}
\def\trace{\operatorname{trace\,}}
\def\Span{\operatorname{span}\,}
\def\dimension{\operatorname{dimension}\,}
\def\codimension{\operatorname{codimension}\,}
\def\nullspace{\scriptk}
\def\kernel{\operatorname{Ker}}
\def\ZZ{ {\mathbb Z} }
\def\p{\partial}
\def\rp{{ ^{-1} }}
\def\Re{\operatorname{Re\,} }
\def\Im{\operatorname{Im\,} }
\def\ov{\overline}
\def\eps{\varepsilon}
\def\lt{L^2}
\def\diver{\operatorname{div}}
\def\curl{\operatorname{curl}}
\def\etta{\eta}
\newcommand{\norm}[1]{ \|  #1 \|}
\def\expect{\mathbb E}
\def\bull{$\bullet$\ }

\def\xone{x_1}
\def\xtwo{x_2}
\def\xq{x_2+x_1^2}
\newcommand{\abr}[1]{ \langle  #1 \rangle}

\newcommand{\Norm}[1]{ \left\|  #1 \right\| }
\newcommand{\set}[1]{ \left\{ #1 \right\} }
\def\one{\mathbf 1}
\def\whole{\mathbf V}
\newcommand{\modulo}[2]{[#1]_{#2}}

\def\scriptf{{\mathcal F}}
\def\scriptg{{\mathcal G}}
\def\scriptm{{\mathcal M}}
\def\scriptb{{\mathcal B}}
\def\scriptc{{\mathcal C}}
\def\scriptt{{\mathcal T}}
\def\scripti{{\mathcal I}}
\def\scripte{{\mathcal E}}
\def\scriptv{{\mathcal V}}
\def\scriptw{{\mathcal W}}
\def\scriptu{{\mathcal U}}
\def\scriptS{{\mathcal S}}
\def\scripta{{\mathcal A}}
\def\scriptr{{\mathcal R}}
\def\scripto{{\mathcal O}}
\def\scripth{{\mathcal H}}
\def\scriptd{{\mathcal D}}
\def\scriptl{{\mathcal L}}
\def\scriptn{{\mathcal N}}
\def\scriptp{{\mathcal P}}
\def\scriptk{{\mathcal K}}
\def\frakv{{\mathfrak V}}

\author{Jian Tan}
\address{Jian Tan
\\
School of Science
\\
Nanjing University of Posts and Telecommunications
\\Nanjing 210023
\\
People's Republic of China
}
\email{tj@njupt.edu.cn; tanjian89@126.com}



\title
[Multilinear fractional integral operators]
{Hardy type spaces estimates for multilinear fractional integral operators}
\maketitle

\begin{abstract}
In this paper, we prove the boundedness of multilinear fractional integral operators from products of Hardy spaces associated with ball quasi-Banach function spaces into their corresponding ball quasi-Banach function spaces. As applications, we establish the boundedness of these operators on various function spaces, including weighted Hardy spaces, variable Hardy spaces, mixed-norm Hardy spaces, Hardy--Lorentz spaces, and Hardy--Orlicz spaces. Notably, several of these results are new, even in special cases, and extend the existing theory of multilinear operators in the context of generalized Hardy spaces.
\end{abstract}

\section{Introduction}
The theory of Hardy spaces originated in the early 20th century from studies in Fourier series and complex analysis in one variable, with a major advancement occurring in 1960 when Stein--Weiss \cite{SW} and Fefferman--Stein \cite{FS} established the theory in the multidimensional setting of $\mathbb R^n$.
The classical Hardy space is characterized by the Littlewood--Paley--Stein square functions, maximal functions, and atomic decompositions, with the latter being a particularly powerful tool in harmonic analysis and wavelet theory for studying function spaces and their associated operators. 
In recent years, the study of multilinear operators in the context of Hardy spaces has gained considerable attention, as evidenced by significant contributions from numerous authors, including those in \cite{CMN,CMN1,GK,HL,LXY}.

On the other hand, the multilinear fractional integral operators represent a natural generalization of their linear counterparts. The genesis of these operators can be traced back to the seminal work of Grafakos (\cite{G}) in 1992, where he introduced and investigated the multilinear fractional integral, defined as follows:
\begin{align*}
\bar{I}_\alpha(\vec{f})(x)=\int_{\mathbb R^n}
\frac{1}{|y|^{n-\alpha}}\prod_{i=1}^mf_i(x-\theta_iy)dy,
\end{align*}
where $\theta_i(i=1,\cdots,m)$ are fixed distinct
and nonzero real numbers and $0<\alpha<n.$
Later on, in 1998, Kenig and Stein \cite{KS} established
the boundedness of another
type of multilinear fractional integral
${I}_{\alpha,A}$
on product of Lebesgue spaces.
${I}_{\alpha,A}(\vec{f})$ is defined by
\begin{align*}
{I}_{\alpha,A}(\vec{f})(x)=\int_{(\mathbb R^n)^m}
\frac{1}{|(y_1,\ldots,y_m)|^{mn-\alpha}}\prod_{i=1}^mf_i(
\ell_i(y_1,\ldots,y_m,x))dy_i,
\end{align*}
where $\ell_i$ is a linear combination of $y_j$'s and $x$
depending on the matrix $A$.
In \cite{LL}, Lin and Lu obtained $I_{\alpha,A}$ is bounded
from product of Hardy spaces to Lebegue spaces when
$\ell_i(y_1, \ldots, y_m, x)=x-y_i$. We denote this multilinear fractional
type integral operators by $\mathcal T_\alpha$, namely,
\begin{align*}
{\mathcal T_\alpha}(f_1,\ldots,f_m)(x)=\int_{(\mathbb R^n)^m}
\frac{1}{|(y_1,\ldots,y_m)|^{mn-\alpha}}\left(\prod_{k=1}^mf_k
(x-y_k)\right)dy_1\cdots dy_m.
\end{align*}
For convenience, we also denote
$$K_\alpha(y_1,\ldots,y_m)=\frac{1}{|(y_1,\ldots,y_m)|^{mn-\alpha}}.$$

Motivated by diverse applications in analysis, many variants of classical Hardy spaces have been extensively investigated, such as weighted Hardy spaces, Hardy-Lorentz spaces, Hardy-Orlicz spaces, and variable Hardy spaces. However, generalized Orlicz spaces and weighted Lebesgue spaces do not necessarily fall into the category of quasi-Banach function spaces. To unify these function spaces under a common framework, Sawano et al. \cite{SHYY} introduced the concept of ball quasi-Banach function spaces $X$, which are defined analogously to quasi-Banach function spaces but with Lebesgue measurable sets replaced by balls in  
$\mathbb{R}^{n}$. 
Then they established the real-variable theory for the corresponding Hardy spaces.
For further details, we refer to \cite{CJY,SHYY,Tan24,YJY} and the references therein.

This leads us to the following natural question:
{\it Can we establish the boundedness of multilinear fractional integral operators on Hardy spaces associated with ball quasi-Banach function spaces?}

The primary objective of this paper is to address this question by investigating the boundedness properties of multilinear fractional integral operators on generalized Hardy-type spaces. We now present the main result of our work, while postponing some technical definitions to Section 2 for brevity.

\begin{theorem}\label{s1t1}
Given an integer $m\ge 1$, 
let $X_1,$ $\ldots,$ $X_m$, $Y_1,$ $\ldots,$ $Y_m$ and $Y$ be ball quasi-Banach function spaces.
Suppose that $X_1,$ $\ldots,$ $X_m$ have the absolutely continuous quasi-norm satisfying Assumptions \ref{ass2.7} and \ref{ass2.8}
and that
$Y$ fulfills Assumption \ref{ass2.8}.
Let $0<p, q, p_k, q_k<\infty$, $\alpha\in (0,mn)$
such that
$$
\frac{1}{q}=\sum_{k=1}^m\frac{1}{q_k},\quad\frac{1}{p}=\sum_{k=1}^m\frac{1}{p_k}
$$ 
and
$$
\frac{1}{p_k}-\frac{1}{q_k}=\frac{\alpha}{mn},$$ 
where $k=1, 2, \cdots, m.$
Suppose that the following H\"older's inequality holds true: for any $f_k\in Y_k$,
\begin{align}\label{holder}
\left\|\prod_{k=1}^mf_k\right\|_{Y}\lesssim \prod_{k=1}^m\|f_k\|_{Y_k}
\end{align}
and that $X_k^{\frac{1}{p_k}}$ and $Y_k^{\frac{1}{q_k}}$ are ball Banach function spaces fulfilling
\begin{align}\label{XY}
\left(Y_k^{\frac{1}{q_k}}\right)'=\left(\left(X_k^{\frac{1}{p_k}}\right)'\right)^{\frac{p_k}{q_k}}.
\end{align}
Further assume that the Hardy--Littlewood maximal operator $\mathcal M$ is bounded on
$\left(Y_k^{\frac{1}{q_k}}\right)'$.
Then $\mathcal T_\alpha$ can be extended to a bounded operator from
$H_{X_1}\times \cdots \times H_{X_m}$ into $Y$.
\end{theorem}

The remainder of this paper is organized as follows. In Section 2, we review essential results on Hardy spaces associated with ball quasi-Banach function spaces that will be utilized throughout the paper. In particular, by employing Rubio de Francia extrapolation, we establish Lemma \ref{l2.7}, which plays a pivotal role in the proof of our main result. 
Section 3 is devoted to applying the main theorem on multilinear fractional integral operators to specific examples of ball quasi-Banach function spaces. These include weighted Lebesgue spaces, variable Lebesgue spaces, mixed-norm Lebesgue spaces, Lorentz spaces, and Orlicz spaces. Notably, to the best of our knowledge, our findings on the boundedness of these multilinear operators on mixed-norm Hardy spaces, Hardy-Lorentz spaces, and Hardy-Orlicz spaces represent entirely new contributions to the field.
In Section 4, we establish the boundedness of multilinear fractional integral operators on Hardy spaces $H_X$ by leveraging the finite atomic characterizations of these spaces. This result further extends our understanding of the behavior of multilinear operators in the context of generalized Hardy spaces.

Throughout this paper, $C$ or $c$ will denote a positive constant that may vary at each occurrence
but is independent to the essential variables, and $A\sim B$ means that there are constants
$C_1>0$ and $C_2>0$ independent of the essential variables such that $C_1B\leq A\leq C_2B$.
Given a measurable set $S\subset \mathbb{R}^n$, $|S|$ denotes the Lebesgue measure and $\chi_S$
means the characteristic function. For a cube $Q$, let $Q^\ast$ denote with the same center
and $2\sqrt{n}$ its side length, i.e. $l(Q^\ast)=100\sqrt{n}l(Q)$.
The symbols $\mathcal S$ and $\mathcal S'$ denote the class of
Schwartz functions and tempered functions, respectively.
As usual, for a function $\psi$ on $\mathbb R^n$,
$\psi_t(x)=t^{-n}\psi(t^{-1}x)$.

\section{Preliminaries}\label{se2}
In this section, we present some notation and known results that will be used throughout the paper.
First, we recall the definitions of ball quasi-Banach function spaces and their related (local) Hardy spaces. For any $x\in\mathbb R^{n}$ and $r\in(0,\infty)$, let $B(x,r):=\{y\in\mathbb R^{n}\colon\vert x-y\vert<r\}$ and
\begin{equation}\label{eq2.1}
\mathbb B(\mathbb R^{n}):=\left\{B(x,r)\colon x\in\mathbb R^{n}\ {\rm and}\ r\in(0,\infty)\right\}.
\end{equation}\par
The concept of ball quasi-Banach function spaces on $\mathbb R^{n}$ is as follows. For more details, see \cite{SHYY}.
\begin{definition}\label{def2.1}
    Let $X\subset\mathscr{M}(\mathbb R^{n})$ be a quasi-normed linear space equipped with a quasi-norm $\| \cdot \|_{X}$ which makes sense for all measurable functions on $\mathbb R^{n}$. Then $X$ is called a \emph{ball\ quasi\text{-}Banach\ function\ space} on $\mathbb R^{n}$ if it satisfies
    \begin{enumerate}[\quad(i)]
    \item if $f\in\mathscr{M}(\mathbb R^{n})$, then $\|f\|_{X}=0$ implies that $f=0$ almost everywhere;
    \item if $f, g\in\mathscr{M}(\mathbb R^{n})$, then $|g|\leq|f|$ almost everywhere implies that $\|g\|_{X}\leq\|f\|_{X}$;
    \item if $\{f_{m}\}_{m\in\mathbb N}\subset\mathscr{M}(\mathbb R^{n})$ and $f\in\mathscr{M}(\mathbb R^{n})$, then $0\leq f_{m}\uparrow f$ almost everywhere as $m\to\infty$ implies that $\|f_{m}\|_{X}\uparrow\|f\|_{X}$ as $m\to
    \infty$;
    \item $B\in\mathbb B(\mathbb R^{n})$ implies that $\chi_{B}\in X$, where $\mathbb B(\mathbb R^{n})$ is the same as (\ref{eq2.1}).
    \end{enumerate}
    Moreover, a ball quasi-Banach function space $X$ is called a \emph{ball\ Banach\ function\ space} if it satisfies
    \begin{enumerate}
    \item[(v)] for any $f,g\in X$
    \[
    \|f+g\|_{X}\leq\|f\|_{X}+\|g\|_{X};
    \]
    \item[(vi)] for any ball $B\in\mathbb B(\mathbb R^{n})$, there exists a positive constant $C_{(B)}$, depending on $B$, such that, for any $f\in X$,
    \[
    \int_{B}\vert f(x)\vert dx\leq C_{(B)}\| f\|_{X}.
    \]
    \end{enumerate}
\end{definition}
The associate space $X^{\prime}$ of any given ball Banach function space $X$ is defined as follows. 
Now we recall the definition of Hardy spaces associated with ball quasi-Banach function spaces in \cite{SHYY}.

\begin{definition}\label{hx}
Let $X$ be a ball quasi-Banach function space and let $f\in \mathcal{S'}$,
$\psi\in \mathcal S$
and $\psi_t(x)=t^{-n}\psi(t^{-1}x)$, $x\in \mathbb{R}^n$.
Denote by $\mathcal{M}_N$ the grand maximal operator given by
$$\mathcal{M}_Nf(x)= \sup\left\{|\psi_t\ast f(x)|: t>0,\psi \in \mathcal{F}_N\right\}$$ for any fixed large integer $N$,
where $$\mathcal{F}_N=\left\{\varphi \in \mathcal{S}:\int\varphi(x)dx=1,\sum_{|\alpha|\leq N}\sup(1+|x|)^N|\partial ^\alpha \varphi(x)|\leq 1\right\}.$$
The Hardy space ${H}_{X}$ is the set of all $f\in \mathcal{S}^\prime$, for which the quantity
$$\|f\|_{{H}_{X}}=\|\mathcal{M}_Nf\|_{X}<\infty.$$
\end{definition}

\begin{definition}\label{def2.2}
    For any given ball Banach function space $X$, its \emph{associate space} (also called the \emph{K\"othe dual space}) $X^{\prime}$ is defined by setting
    \[
    X^{\prime}:=\{f\in\mathscr{M}(\mathbb R^{n})\colon\|f\|_{X^{\prime}}<\infty\},
    \]
    where, for any $f\in X^{\prime}$,
    \[
    \|f\|_{X^{\prime}}:=\sup\left\{ \|fg\|_{L^{1}}\colon g\in X,\ \|g\|_{X}=1 \right\},
    \]
    and $\|\cdot\|_{X^{\prime}}$ is called the \emph{associate norm} of $\|\cdot\|_{X}$.
\end{definition}

Recall that $X$ is said to have an \emph{absolutely continuous quasi-norm} if, for any $f\in X$ and any measurable subsets $\{E_{j}\}_{j\in\mathbb N}\subset\mathbb R^{n}$ with both $E_{j+1}\subset E_{j}$ for any $j\in\mathbb N$ and $\bigcap_{j\in\mathbb N}E_{j}=\emptyset$,  $ \|f\chi_{E_{j}} \|_{X}\downarrow0$ as $j\to\infty$.
Then we give the definition of the $(X,q,d)$-atom.
\begin{definition}\label{def2.6}
    Let $X$ be a ball quasi-Banach function space and $q\in(1,\infty]$. Assume that $d\in\mathbb Z_{+}$. Then a measurable function $a$ is called a \emph{$(X,q,d)$-atom} if 
    \begin{enumerate}[(i)]
        \item there exists a cube $Q\subset\mathbb R^{n}$ such that ${\rm supp}a\subset Q$;
        \item $\|a\|_{L^{q}}\leq\frac{\vert Q \vert^{1/q}}{\|\chi_{Q}\|_{X}}$;
        \item $\int_{\mathbb R^{n}}a(x)x^{\alpha}dx=0$ for any multi-index $\alpha\in\mathbb Z^{n}_{+}$ with $\vert\alpha\vert\leq d$.
    \end{enumerate}
\end{definition}

Denote by $L^{1}_{loc}(\mathbb R^{n})$ the set of all locally integral functions on $\mathbb R^{n}$. Recall that the \emph{Hardy--Littlewood maximal operator} $\mathcal{M}$ is defined by setting, for any measurable function $f$ and any $x\in\mathbb R^{n}$,
\[
\mathcal{M}(f)(x)=\sup\limits_{x\in Q}\frac{1}{\vert Q\vert}\int_{Q}f(u)du.
\]
For any $\theta\in \left(0,\infty\right)$, the {\it powered Hardy--Littlewood maximal operator} $\mathcal M^{\left(\theta\right)}$ is defined by setting, for all $f\in L_{{\rm loc}}^{1}\left(\mathbb{R}^{n}\right)$ and $x\in \mathbb{R}^{n}$,
\begin{align}\label{eq:the powered Hardy--Littlewood maximal operator}
	\mathcal M^{\left(\theta\right)}\left(f\right)\left(x\right) := \left\{\mathcal M\left(\left|f\right|^{\theta}\right)\left(x\right) \right\}^{1/\theta}.
\end{align}

Moreover, we also need two basic assumptions on $X$ as follows. 
\begin{assumption}\label{ass2.7}
Let $X$ be a ball quasi-Banach function space. For some $\theta,\ s\in(0,1]$ and $\theta<s$, there exists a positive constant $C$ such that, for any $\{f_{j}\}_{j=1}^{\infty}\subset L^{1}_{loc}(\mathbb R^{n})$,
\begin{align}\label{2.7}
\left\|  \left\{  \sum_{j=1}^{\infty}\left[ \mathcal{M}^{(\theta)}(f_{j}) \right]^{s} \right\}^{\frac{1}{s}}  \right\|_{X}\leq C\left\|  \left\{  \sum_{j=1}^{\infty}\vert f_{j} \vert^{s} \right\}^{\frac{1}{s}} \right\|_{X}.
\end{align}
\end{assumption}

\begin{assumption}\label{ass2.8}  Let $X$ be a ball quasi-Banach function space.
Fix $q>1$. Suppose that $0<s<q$. For any $f\in(X^{1/s})^{\prime}$,
    \[
    \left\| \mathcal{M}^{((q/s)^{\prime})}(f) \right\|_{(X^{1/s})^{\prime}}\leq C\|f\|_{(X^{1/s})^{\prime}}.
    \]
\end{assumption}

The following lemmas play a key role in the proof of the main result.

\begin{lemma}\label{l2.7}\cite[Lemma 2.8]{CT}
 Assume that $X$ be a ball quasi-Banach function space.
Let $s\in(0,1]$ such that $X^{1/s}$ is a ball Banach function and $X$ satisfies Assumption \ref{ass2.8} for some $q\in\left[1,\infty \right)$.    
    For all sequences of cubes $\{Q_{j}\}_{j=1}^{\infty}$ and non-negative functions $\{g_{j}\}_{j=1}^{\infty}$, if~$ \sum_{j=1}^{\infty}\chi_{Q_{j}}g_{j}\in X$, then
    \[
    \left\|  \sum_{j=1}^{\infty}\chi_{Q_{j}}g_{j} \right\|_{X}\leq C\left\| \sum_{j=1}^{\infty}\left( \frac{1}{\vert Q_{j}\vert}\int_{Q_{j}}g_{j}^{q}(y)dy \right)^{\frac{1}{q}}\chi_{Q_{j}} \right\|_{X}.
    \]
\end{lemma}

\begin{lemma}\label{l2.8}
Let both $X$ and $Y$ be ball quasi-Banach function spaces. Let $\alpha \in (0, n)$ and $0 < p_0 < q_0 \leq 1$ be such that 
\[
\frac{1}{q_0} = \frac{1}{p_0} - \frac{\alpha}{n}
\]
and that $X^{\frac{1}{p_0}}$ and $Y^{\frac{1}{q_0}}$ are ball Banach function spaces, and
\[
\left(Y^{\frac{1}{q_0}}\right)' = \left(\left(X^{\frac{1}{p_0}}\right)'\right)^{\frac{p_0}{q_0}}.
\]
Further assume that the Hardy--Littlewood maximal operator $\mathcal M$ is bounded on
$\left(Y^{\frac{1}{q_0}}\right)'$.
Then
    \[
    \left\|  \sum_{j=1}^{\infty}\lambda_j|Q_j|^{\frac{\alpha}{n}}\chi_{Q^\ast_{j}}\right\|_{Y}\leq 
    C\left\| \sum_{j=1}^{\infty}\lambda_j\chi_{Q_{j}}\right\|_{X},
    \]
    where $\lambda_j> 0$ and $Q_j$
is any sequence of cubes in $\mathbb R^n$.
\end{lemma}

In order to prove this lemma, we need recall some classical definitions about weights. Suppose that a weight $\omega$ is a non-negative, locally integrable function such that $0<\omega(x)<\infty$ for almost every $x\in\mathbb R^{n}$. It is said that $\omega$ is in the \emph{Muckenhoupt class} $A_{p}$ for $1<p<\infty$ if
\[
[\omega]_{A_{p}}=\sup\limits_{Q}\left(\frac{1}{Q}\int_{Q}\omega(x)dx\right)\left(\frac{1}{Q}\int_{Q}\omega(x)^{-\frac{1}{p-1}}dx\right)^{p-1}<\infty,
\]
where $Q$ is any cube in $\mathbb {R}^{n}$ and when $p=1$, a weight $\omega\in A_{1}$ if for almost everywhere $x\in\mathbb {R}^{n}$,
\[
\mathcal{M}\omega(x)\leq C\omega(x),
\]
Therefore, define the set
\[
A_{\infty}=\bigcup\limits_{1\leq p<\infty}A_{p}.
\]
Given a weight $\omega\in A_{\infty}$, define $$q_{\omega}=\inf\{q\geq1\colon \omega\in A_{q}\}.$$

\begin{proof}[Proof of Lemma \ref{l2.8}]
    For any non-negative $h\in (Y^{\frac{1}{q_0}})'$, define the Rubio de Francia iteration algorithm by
    \[
    \mathcal{R}h(x):=\sum_{k\in\mathbb Z_{+}}\frac{\mathcal{M}^{k}(h)(x)}{2^{k}\|\mathcal{M}\|^{k}_{(Y^{\frac{1}{q_0}})'\rightarrow (Y^{\frac{1}{q_0}})'}},
    \]
    where $\mathcal{M}^{0}(h):=\vert h\vert$ and, for any $k\in\mathbb N$, $\mathcal{M}^{k}:=\mathcal{M}\circ\dots\circ\mathcal{M}$ is $k$ iterations of $\mathcal{M}$. Then, $\mathcal{R}h$ has the following properties:
    \begin{enumerate}[(i)]
    \item $h(x)\leq\mathcal{R}h(x)$;
    \item $\|\mathcal{R}h\|_{(Y^{\frac{1}{q_0}})'}\leq 2\|h\|_{(Y^{\frac{1}{q_0}})'}$;
    \item $\mathcal{M}(\mathcal{R}h)\leq 2\|\mathcal{M}\|_{(Y^{\frac{1}{q_0}})'\rightarrow (Y^{\frac{1}{q_0}})'}\mathcal{R}h$, namely, $\mathcal{R}h\in A_{1}$ and
    \[
    [\mathcal{R}h]_{A_{1}}\leq 2\|\mathcal{M}\|_{(Y^{\frac{1}{q_0}})'\rightarrow (Y^{\frac{1}{q_0}})'}.\]
    \end{enumerate}\par
 Then by duality, we deduce that
    \[
    \begin{aligned}
        \left\|\sum_{j=1}^{\infty}\lambda_j|Q_j|^{\frac{\alpha}{n}}\chi_{Q^\ast_{j}}\right\|_{Y}^{q_0}
        =&\sup_{\|h\|_{(Y^{1/q_0})^{\prime}}=1}\int_{\mathbb R^{n}}\left| \sum_{j=1}^{\infty}\lambda_j|Q_j|^{\frac{\alpha}{n}}\chi_{Q^\ast_{j}}\right|^{q_0}h(x) dx\\
        &\leq\sup_{\|h\|_{(Y^{1/q_0})^{\prime}}=1}\int_{\mathbb R^{n}}\left| \sum_{j=1}^{\infty}\lambda_j|Q_j|^{\frac{\alpha}{n}}\chi_{Q^\ast_{j}}\right|^{q_0}\mathcal{R}h(x) dx.
    \end{aligned} 
    \]
By using \cite[Lemma 2.1]{CMN1} with $w:=(\mathcal Rh)^{\frac{p_0}{q_0}}$ and the above properties, we conclude that
    \[
    \begin{aligned}
&\int_{\mathbb R^{n}}\left| \sum_{j=1}^{\infty}\lambda_j|Q_j|^{\frac{\alpha}{n}}\chi_{Q^\ast_{j}}\right|^{q_0}\mathcal{R}h(x) dx
=\left\|  \sum_{j=1}^{\infty}\lambda_j|Q_j|^{\frac{\alpha}{n}}\chi_{Q^\ast_{j}}\right\|^{q_0}_{L^{q_0}_{w^{\frac{q_0}{p_0}}}}\\
&\le C\left\| \sum_{j=1}^{\infty}\lambda_j\chi_{Q_{j}}\right\|^{q_0}_{L^{p_0}_{w}}=C\left\| \sum_{j=1}^{\infty}\lambda_j\chi_{Q_{j}}\right\|^{q_0}_{L^{p_0}_{(\mathcal Rh)^{\frac{p_0}{q_0}}}}.
    \end{aligned}    \]
    
Moreover, we observe that
\begin{align*}
&\left\|\sum_{j=1}^{\infty}\lambda_j\chi_{Q_{j}}\right\|_
{L^{p_0}_{(\mathcal Rh)^{\frac{p_0}{q_0}}}}
\leq
\left\| \left[\sum_{j=1}^{\infty}\lambda_j\chi_{Q_{j}}\right]^{p_0} \right\|_{X^{1/p_0}}
\left\| (\mathcal{R}h)^{p_0/q_0} \right\|_{(X^{1/p_0})'}\\
&=
\left\|\sum_{j=1}^{\infty}\lambda_j\chi_{Q_{j}}\right\|_{X} \left\|\mathcal{R}h\right\|^{1/q_0}_{[(X^{1/p_0})']^{p_0/q_0}}
\lesssim
\left\|\sum_{j=1}^{\infty}\lambda_j\chi_{Q_{j}}\right\|_{X}
\| h \|^{1/q_0}_{(Y^{1/q_0})'}.
\end{align*}

 Therefore, we obtain that
     \[
    \begin{aligned}
        \left\|\sum_{j=1}^{\infty}\lambda_j|Q_j|^{\frac{\alpha}{n}}\chi_{Q^\ast_{j}}\right\|_{Y}
        &\leq\sup_{\|h\|_{(Y^{1/q_0})^{\prime}}=1}\left\{\int_{\mathbb R^{n}}\left| \sum_{j=1}^{\infty}\lambda_j|Q_j|^{\frac{\alpha}{n}}\chi_{Q^\ast_{j}}\right|^{q_0}\mathcal{R}h(x) dx\right\}^{\frac{1}{q_0}}\\
&\leq\sup_{\|h\|_{(Y^{1/q_0})^{\prime}}=1}\left\|\sum_{j=1}^{\infty}\lambda_j\chi_{Q_{j}}\right\|_{X}
\| h \|^{1/q_0}_{(Y^{1/q_0})'}\\
&\leq\left\|\sum_{j=1}^{\infty}\lambda_j\chi_{Q_{j}}\right\|_{X}.
    \end{aligned} 
    \]
Thus, we have completed the proof of this lemma.
\end{proof}


\begin{definition}\label{def4.4}
    Let $X$ be a ball quasi-Banach function space satisfying Assumptions~{\ref{ass2.7}} and {\ref{ass2.8}}. 
    Let $1<q\leq\infty$ and $d\in\mathbb Z_{+}$. The \emph{finite atomic Hardy space} $H_{fin}^{X,q,d}(\mathbb R^{n})$ associated with $X$, is defined by
    \[
    H_{fin}^{X,q,d}(\mathbb R^{n})=\left\{ f\in\mathcal{S}^{\prime}(\mathbb R^{n})\colon f=\sum_{j=1}^{M}\lambda_{j}a_{j} \right\},
    \]
    where $\{a_{j}\}_{j=1}^{M}$ are $(X,q,d)$-atoms satisfying
    \[
    \left\| \sum_{j=1}^{M} \frac{\lambda_{j}\chi_{Q_{j}}}{\|\chi_{Q_{j}}\|_{X}}   \right\|_{X}<\infty.
    \]
    Furthermore, the quasi-norm $\|\cdot\|_{H_{fin}^{X,q,d}}$ in $H_{fin}^{X,q,d}(\mathbb R^{n})$ is defined by setting, for any $f\in H_{fin}^{X,q,d}(\mathbb R^{n})$,
    \[
    \|f\|_{H_{fin}^{X,q,d}}:=\inf\left\{\left\| \sum_{j=1}^{M} \frac{\lambda_{j}\chi_{Q_{j}}}{\|\chi_{Q_{j}}\|_{X}}   \right\|_{X}\right\},
    \]
    where the infimum is taken over all finite decomposition of $f$.
\end{definition}

In what follows, the symbol $\mathcal C(\mathbb R^{n})$ is defined to be the set of all
continuous complex-valued functions on $\mathbb R^n$.
From \cite[Theorem 1.10]{YYY} and \cite[Corollary 4.3]{CT},
we immediately obtain the following theorem.
\begin{theorem}\label{finite}
    Let $X$ be a ball quasi-Banach function space satisfying Assumptions~{\ref{ass2.7}} and {\ref{ass2.8}}. 
Further assume that $X$ has an absolutely continuous quasi-norm. Fix $q\in(1,\infty)$ and $d\in\mathbb{Z}_{+}$. For $f\in H_{fin}^{X,q,d}(\mathbb R^{n})$,
    \[
    \|f\|_{H_{fin}^{X,q,d}}\sim\|f\|_{H_{X}}.
    \]
Moreover,
for $f\in H_{fin}^{X,\infty,d}(\mathbb R^{n})\cap \mathcal C(\mathbb R^{n})$,
    \[
    \|f\|_{H_{fin}^{X,\infty,d}}\sim\|f\|_{H_{X}}.
    \]
\end{theorem}

\section{Examples and applications}

In this section, we will apply the main theorem on multilinear fractional integral operators to several concrete examples of ball quasi-Banach function spaces, namely, weighted Lebesgue spaces (Subsection \ref{s3s1}), variable Lebesgue spaces (Subsection \ref{s3s2}),
mixed-norm Lebesgue spaces (Subsection \ref{s3s3}), Lorentz spaces (Subsection \ref{s3s4}) and Orlicz spaces (Subsection \ref{s3s5}).
To our best knowledge, our results on the boundedness of these multilinear operators on mixed-norm Hardy spaces, Hardy--Lorentz spaces and Hardy--Orlicz spaces are completely new. Given the generality of our results, further applications are both expected and foreseeable.

\subsection{Weighted Hardy spaces}\label{s3s1}

From \cite[Subsection 7.1]{SHYY} we know that $L^{p}_{\omega}(\mathbb R^{n})$ with $p\in(0,\infty)$ and $\omega\in A_{\infty}(\mathbb R^{n})$ is a ball quasi-Banach function space. Moreover, 
by \cite[Remark 2.4(b)]{WYY} we know that Assumption~\ref{ass2.7} holds true when $\theta,\ s\in(0,1]$, $\theta<s$ and $X:=L^{p}_{\omega}(\mathbb R^{n})$ with $p\in(\theta,\infty)$ and $\omega\in A_{p/\theta}(\mathbb R^{n})$.
By \cite[Remark 2.7(b)]{WYY}, for any \( s \in (0, \min \{1, p\}) \), \( w \in A_{p/s} (\mathbb{R}^n) \) and \( q \in (\max\{1, p\}, \infty] \) large enough such that \( w^{1 - (p/s)'} \in A_{(p/s)' / (q/s)'} (\mathbb{R}^n) \), it holds true that \( (p/s)' / (q/s)' > 1 \) and
\[
\left[ (X^{1/s})' \right]^{1/(q/s)'} = L_{w^{1 - (p/s)'}}^{(p/s)' / (q/s)'} (\mathbb{R}^n),
\]
which, together with the boundedness of \( M \) on the weighted Lebesgue space, further implies that Assumption~\ref{ass2.8} holds true. Suppose that $q_k$ are such that $p_k\le q_k<\infty,$ $k=1,\cdots,m.$
In this weighted case, we have $Y=L^{q}_{\bar \omega}(\mathbb R^{n})$, $Y_1=L^{q_1}_{\omega_1^{q_1/p_1}}(\mathbb R^{n})$, $\dots,$
$Y_m=L^{q_m}_{\omega_m^{q_m/p_m}}(\mathbb R^{n})$, where $\bar\omega=\prod_{k=1}^m\omega_k^{{q}/{p_k}}$ and
$$\frac{1}{q}=\sum_{k=1}^m\frac{1}{q_k},$$
for $k=1,\dots, m.$
Then by classical H\"older's inequality, we can get that for all $f_k\in L^{p_k}_{\omega_k}$, 
$$\left\|\prod_{k=1}^mf_k\right\|_{L^{q}_{\bar\omega}}\le \prod_{k=1}^m\|f_k\|_{L^{q_k}_{\omega_k^{q_k/p_k}}}.$$
Meanwhile, 
let $p_k^0<p_k$ such that $q_k^0<q_k$, then
$L_{\omega_k}^{\frac{p_k}{p^0_k}}$ and $L_{\omega_k}^{\frac{q_k}{q^0_k}}$ are ball Banach function spaces fulfilling
\begin{align*}
\left(L_{\omega_k}^{\frac{q_k}{q^0_k}}\right)'
=\left(\left([L_{\omega_k}^{p_k}]^{\frac{1}{p^0_k}}\right)'\right)^{\frac{p^0_k}{q^0_k}}.
\end{align*}
From these results and Theorems \ref{s1t1}, it follows the
following conclusion on weighted Hardy spaces.

\begin{theorem}\label{th51}\quad Let $0<\alpha<mn$.
Given an integer $m\ge 1$, $p_1,\dots,p_m\in (0,\infty)$, 
define $0<p<\infty$ by 
$$
\frac{1}{p}=\frac{1}{p_1}+\cdots+\frac{1}{p_m}>\frac{\alpha}{n}
$$
and define
$0<q<\infty$ by 
$$
\frac{1}{q}=\frac{1}{p}-\frac{\alpha}{n}.
$$
Suppose that 
$q_k$ are such that $p_k<q_k<\infty$, $1\le k\le m$, and
$$\frac{1}{q}=\sum_{k=1}^m\frac{1}{q_k},$$
such that
${\omega_k}\in RH_{{q_k}/{p_k}}$. If 
$\mathcal T$ is an multilinear fractional integral operator, 
then $\mathcal T$ can be extended to a bounded operator from
$H^{p_1}_{\omega_1}\times\cdots\times H^{p_m}_{\omega_m}$ into $L^{q}_{\bar \omega}$,
where $\bar\omega=\prod_{k=1}^m\omega_k^{{q}/{p_k}}$.\end{theorem}

We also note that Cruz-Uribe et al. have established the above boundedness of the multilinear fractional fractional intergral operators on weighted Hardy spaces in \cite{CMN1}.

\subsection{Variable Hardy spaces}\label{s3s2}
The variable exponent function spaces,
such as the variable Lebesgue spaces and the variable
Sobolev spaces, were studied by a substantial number
of researchers (see, for instance, \cite{CFMP,KR}).
For any Lebesgue measurable function $p(\cdot):
\mathbb R^n\rightarrow (0,\infty]$ and for any
measurable subset $E\subset \mathbb{R}^n$, we denote
$p^-(E)= \inf_{x\in E}p(x)$ and $p^+(E)= \sup_{x\in E}p(x).$
Especially, we denote $p^-=p^{-}(\mathbb{R}^n)$ and $p^+=p^{+}(\mathbb{R}^n)$.
Let $p(\cdot)$: $\mathbb{R}^n\rightarrow(0,\infty)$ be a measurable
function with $0<p^-\leq p^+ <\infty$ and $\mathcal{P}^0$
be the set of all these $p(\cdot)$.
Let $\mathcal{P}$ denote the set of all measurable functions
$p(\cdot):\mathbb{R}^n \rightarrow[1,\infty) $ such that
$1<p^-\leq p^+ <\infty.$

\begin{definition}\label{s1de1}\quad
Let $p(\cdot):\mathbb R^n\rightarrow (0,\infty]$
be a Lebesgue measurable function.
The {\it variable Lebesgue space} $L^{p(\cdot)}$ consisits of all
Lebesgue measurable functions $f$, for which the quantity
$\int_{\mathbb{R}^n}|\varepsilon f(x)|^{p(x)}dx$ is finite for some
$\varepsilon>0$ and
$$\|f\|_{L^{p(\cdot)}}=\inf{\left\{\lambda>0: \int_{\mathbb{R}^n}\left(\frac{|f(x)|}{\lambda}\right)^{p(x)}dx\leq 1 \right\}}.$$
\end{definition}
As a special case of the theory of Nakano and Luxemberg, we see that $L^{p(\cdot)}$
is a quasi-normed space. Especially, when $p^-\geq1$, $L^{p(\cdot)}$ is a Banach space. In the study of variable exponent function spaces it is common
to assume that the exponent function $p(\cdot)$ satisfies $LH$
condition.
We say that $p(\cdot)\in LH$, if $p(\cdot)$ satisfies

 $$|p(x)-p(y)|\leq \frac{C}{-\log(|x-y|)} ,\quad |x-y| \leq 1/2$$
and
 $$|p(x)-p(y)|\leq \frac{C}{\log(|x|+e)} ,\quad |y|\geq |x|.$$

 In fact, from \cite[Subsection 7.4]{SHYY} we know that whenever $p(\cdot)\in\mathcal{P}^{0}$, $L^{p(\cdot)}(\mathbb R^{n})$ is a ball quasi-Banach function space. Let $X:=L^{p(\cdot)}(\mathbb R^{n})$ with $p(\cdot)\in LH$. By \cite[Remark 2.4 (g)]{WYY}, we know that Assumption~\ref{ass2.7} holds true when $s\in(0,1]$ and $\theta\in(0,\min\{s,p_{-}\})$. Besides, let $X:=L^{p(\cdot)}(\mathbb R^{n})$ with $p(\cdot)\in LH$ and $0<p^{-}\leq p^{+}<\infty$. From \cite[Remark 2.7 (f)]{WYY}, Assumption~\ref{ass2.8} holds true when $s\in(0,\min\{1,p^{-}\})$ and $q\in(\max\{1,p^{+}\},\infty]$. Moreover,
the following generalized H\"{o}lder inequality on variable Lebesgue spaces
can be found in in \cite{CF, TLZ}.

\begin{proposition}
Given exponent function $q_k(\cdot)\in \mathcal{P}^0,$ define
$q(\cdot)\in \mathcal{P}^0$ by
$$\frac{1}{q(x)}=\sum_{k=1}^m\frac{1}{q_k(x)},$$
where $k=1,\dots, m.$
Then for all $f_k\in L^{q_k(\cdot)}(\mathbb R^{n})$ we have
$$\left\|\prod_{k=1}^mf_k\right\|_{L^{q(\cdot)}}\lesssim \prod_{k=1}^m\|f_k\|_{L^{q_k(\cdot)}}.$$
\end{proposition}
Furthermore, 
 $(L^{p_k(\cdot)})^{\frac{1}{p^-_k}}$ and $(L^{q_k(\cdot)})^{\frac{1}{q^-_k}}$ are ball Banach function spaces fulfilling
\begin{align}\label{XY}
\left(L^{\frac{q_k(\cdot)}{q^-_k}}\right)'=\left(\left(L^{\frac{p_k(\cdot)}{p^-_k}}\right)'\right)^{\frac{p^-_k}{q^-_k}}.
\end{align}

Therefore, Theorem~\ref{s1t1} holds true in the variable exponents settings, where $Y=L^{q(\cdot)}(\mathbb R^{n})$, $X_1=L^{p_1(\cdot)}(\mathbb R^{n})$, $\dots,$
$X_m=L^{p_m(\cdot)}(\mathbb R^{n})$. 
Denote $H^{q(\cdot)}(\mathbb R^{n})$ the variable Hardy space as in Definition \ref{hx} with $X$ replaced by $L^{q(\cdot)}(\mathbb R^{n})$, which is introduced in \cite{CW, NS}.
We remark that these results in this subsection have been obtained
in \cite{CMN1,Tan2020}. For similar results, see for example \cite{CMN,Tan,Tan1,TZ1}.

\begin{theorem}\label{th54} 
Given an integer \( m \geq 1 \)
and $0<\alpha<mn$, let \( p_1(\cdot), \dots, p_m(\cdot)\in LH \) such that \( 0 < (p_k)^- \leq (p_k)^+ < \infty \)
and
\[
\frac{1}{[p_1(\cdot)]^+} + \cdots + \frac{1}{[p_m(\cdot)]^+}>\frac{\gamma}{n}.
\]
Define
\[
\frac{1}{q(\cdot)} = \frac{1}{p_1(\cdot)} + \cdots + \frac{1}{p_m(\cdot)}-\frac{\alpha}{n}.
\]
If 
$\mathcal T$ is an multilinear fractional integral operator, 
then $\mathcal T$ can be extended to a bounded operator from
$H^{p_1(\cdot)} \times \cdots \times H^{p_m(\cdot)}$ into $L^{q(\cdot)}$.
\end{theorem}

\subsection{Mixed-norm Hardy spaces}\label{s3s3}
Recall the definition of mixed-norm Lebesgue space $L^{\Vec{p}}(\mathbb R^{n})$ in \cite{BP,HLYY,HY} and the references therein.\par
\begin{definition}
    Let $\Vec{p}:=(p_{1}, \cdots,p_{n})\in(0,\infty]^{n}$. The \emph{mixed-norm Lebesgue space} $L^{\Vec{p}}(\mathbb R^{n})$ is defined to be the set of all the measurable functions $f$ on $\mathbb R^{n}$ such that
    \[
    \|f\|_{L^{\Vec{p}}}:=\left\{  \int_{\mathbb R^{n}}\cdots\left[ \int_{\mathbb R^{n}}\vert f(x_{1},\cdots,x_{n})\vert^{p_{1}}dx_{1}  \right]^{\frac{p_{2}}{p_{1}}}\cdots dx_{n}\right\}^{\frac{1}{p_{n}}}<\infty
    \]
    with the usual modifications made when $p_{i}=\infty$ for some $i\in\{1,\cdots,n\}$. 
\end{definition}
    For any exponent vector $\Vec{p}:=(p_{1}, \cdots,p_{n})$, let
    \[
    p_{-}:=\min\{p_{1},\dots,p_{n}\}\quad{\rm and}\quad p_{+}:=\max\{p_{1},\dots,p_{n}\}.
    \] 
   Similarly, we denote  \[
    p^k_{-}:=\min\{p^k_{1},\dots,p^k_{n}\}\quad{\rm and}\quad p^k_{+}:=\max\{p^k_{1},\dots,p^k_{n}\},\]
    for $k=1,2,\dots,m$.
From the definition of the mixed-norm Lebesgue space, we know that $L^{\Vec{p}}(\mathbb R^{n})$ with $\Vec{p}\in(0,\infty)^{n}$ is a ball quasi-Banach function space. By \cite{HLYY}, Assumption~\ref{ass2.7} holds true when $s\in(0,1]$, $\theta\in(0,\min\{s,p_{-}\})$ and $X:=L^{\Vec{p}}(\mathbb R^{n})$ with $\Vec{p}\in(0,\infty)^{n}$. Meanwhile, the 
Assumption~\ref{ass2.8} holds true when $X:=L^{\Vec{p}}(\mathbb R^{n})$ with $\Vec{p}\in(0,\infty)^{n}$, $s\in(0,p_{-})$ and $q\in(\max\{1,p_{+}\},\infty]$. Furthermore, $L^{\Vec{p}}(\mathbb R^{n})$ has an absolutely continuous quasi-norm. 
Moreover, by the definition of mixed-norm Lebesgue spaces and
the classical H\"{o}lder inequality on Lebesgue spaces,
a simple computation yields the following proposition. 

\begin{proposition}\label{mixholder}
Let $0<p_i, p_i^k<\infty$ such that
$$\frac{1}{p_i}=\sum_{k=1}^m\frac{1}{p_i^k},$$
where $i=1,\dots,n $ and $k=1,\dots, m.$
Then for all $f_k\in L^{\Vec{p_k}}(\mathbb R^{n})$ we have
$$\left\|\prod_{k=1}^mf_k\right\|_{L^{\Vec{p}}}\lesssim \prod_{k=1}^m\|f_k\|_{L^{\Vec{p_k}}},$$
where $\Vec{p_k}:=(p^k_{1}, \cdots,p^k_{n})$.
\end{proposition}
Furthermore, we set $\Vec{q_k}:=\left(\frac{np_1^k}{n-\alpha p_1^k},\cdots,\frac{np_n^k}{n-\alpha p_n^k}\right)\in (0,\infty)^n$.
Then we can choose $q_0\le q_-^k$ and $p_0\le p_-^k$ such that $(L^{\Vec{q_k}})^{\frac{1}{q_0}}$ and $(L^{\Vec{p_k}})^{\frac{1}{p_0}}$ are ball Banach function spaces fulfilling
\begin{align*}
\left([L^{\Vec{q_k}}]^{\frac{1}{q_0}}\right)'= 
\left(\left([L^{\Vec{p_k}}]^{\frac{1}{p_0}}\right)'\right)^{\frac{p_0}{q_0}}.
\end{align*}
 
Therefore, Theorem~\ref{s1t1} holds true when $Y= L^{\Vec{q}}(\mathbb R^{n})$, 
$X_1={L^{\Vec{p_1}}}(\mathbb R^{n})$, $\dots,$
$X_m={L^{\Vec{p_m}}}(\mathbb R^{n})$.
Denote $H^{\Vec{p}}(\mathbb R^{n})$ the mixed-norm Hardy space as in Definition \ref{hx} with $X$ replaced by $L^{\Vec{p}}(\mathbb R^{n})$, which can be found in \cite{CJY,HLYY,HY}.

\begin{theorem}\label{th510} 
Given an integer \( m \geq 1 \)
and $0<\alpha<mn$.
Let $0<p_i, p_i^k,q_i,q_i^k<\infty$ such that
such that
$$\frac{1}{p_i}=\sum_{k=1}^m\frac{1}{p_i^k},\quad
\frac{1}{q_i}=\sum_{k=1}^m\frac{1}{q_i^k}$$
where $k=1,2,\ldots,m$. 
Suppose that $\Vec{q_k}:=\left(\frac{mnp_1^k}{mn-\alpha p_1^k},\cdots,\frac{mnp_n^k}{mn-\alpha p_n^k}\right)\in (0,\infty)^n$.
If 
$\mathcal T$ is an multilinear fractional integral operator, 
then $\mathcal T$ can be extended to a bounded operator from
$H^{\Vec{p_1}}\times \cdots \times H^{\Vec{p_m}}$ into $L^{\Vec{q}}$.
\end{theorem}

\subsection{Hardy--Lorentz spaces}\label{s3s4}
We recall some basic results on the Lorentz space in \cite{Lo, SHYY,WYY}.
\begin{definition}
    The \emph{Lorentz space} $L^{p,q}(\mathbb R^{n})$ is defined to be the set of all measurable functions $f$ on $\mathbb R^{n}$ such that, when $p,q\in(0,\infty)$,
    \[
    \|f\|_{L^{p,q}}:=\left\{  \int_{0}^{\infty}\left[ t^{\frac{1}{p}}f^{*}(t) \right]^{q}\frac{dt}{t} \right\}^{\frac{1}{q}}<\infty,
    \]
    and, when $p\in(0,\infty)$ and $q=\infty$,
    \[
    \|f\|_{L^{p,q}}:=\sup_{t\in(0,\infty)}t^{\frac{1}{p}}f^{*}(t)<\infty,
    \]
    where $f^{*}$ denotes the decreasing rearrangement of $f$, which is defined by setting, for any $t\in[0,\infty)$,
    \[
    f^{*}(t):=\inf\{s\in(0,\infty)\colon \mu_{f}(s)\leq t\}
    \]
    with $\mu_{f}(s):=\left\vert  \left\{  x\in\mathbb R^{n}\colon\vert f(x)\vert>s \right\} \right\vert$.
\end{definition}

It is known  in \cite{SHYY} that $L^{p,q}(\mathbb R^{n})$ is a ball Banach function space when $p,q\in(1,\infty)$ or $p\in(1,\infty)$ and $q=\infty$
and is a ball quasi-Banach function space when $p,q\in(0,\infty)$ or $p\in(0,\infty)$ and $q=\infty$. 
From \cite[Theorem 2.3(iii)]{CRS}, Assumption~\ref{ass2.7} holds true when  $s\in(0,1]$ and $\theta\in(0,\min\{s,p,q\})$. 
If $X:=L^{p,r}(\mathbb R^{n})$ with $p\in(0,\infty)$ and $r\in(0,\infty]$. 
By \cite[Theorem 1.4.16]{G} and \cite[Remark 2.7(d)]{WYY},
Assumption~\ref{ass2.8} holds true when $s\in(0,\min\{p,r\})$ and $q\in(\max\{1,p,r\},\infty]$. 

Moreover, we also need the following H\"{o}lder inequality on Lorentz spaces which can be found in \cite[Section 4.3.2]{HN}.

\begin{proposition}\label{mixlorentz}
Let $0<p, q, p_k,q_k<\infty$ such that
$$
\frac{1}{p}=\sum_{k=1}^m\frac{1}{p_k},\quad
\frac{1}{q}=\sum_{k=1}^m\frac{1}{q_k},
$$
where $k=1,\dots, m.$
Then for all $f_k\in L^{p,q}(\mathbb R^{n})$ we have
$$\left\|\prod_{k=1}^mf_k\right\|_{L^{p,q}}\lesssim \prod_{k=1}^m\|f_k\|_{L^{p_k,q_k}}.$$
\end{proposition}
 
Moreover, we can choose that $q_0<\min\{\tilde p, q\}$ and $p_0<\min\{p,q\}$ such that
$
\frac{1}{q_0}=\frac{1}{p_0}-\frac{\alpha}{n}$ and $(L^{\tilde p, q})^{\frac{1}{q_0}}$ and $(L^{p,q})^{\frac{1}{p_0}}$ are ball Banach function spaces fulfilling
\begin{align*}
\left([L^{\tilde p,q}]^{\frac{1}{q_0}}\right)'= 
\left(\left([L^{p,q}]^{\frac{1}{p_0}}\right)'\right)^{\frac{p_0}{q_0}}.
\end{align*}
 
Therefore, Theorem~\ref{s1t1} holds true when $X= L^{p,q}(\mathbb R^{n})$, 
$X_1={L^{p_1,q_1}}(\mathbb R^{n})$, $\dots,$
$X_m={L^{p_m,q_m}}(\mathbb R^{n})$.
Denote $H^{p,q}(\mathbb R^{n})$ the Hardy--Lorentz space as in Definition \ref{hx} with $X$ replaced by $L^{p,q}(\mathbb R^{n})$. See \cite{SHYY,WYY}
for more details on the Hardy--Lorentz space.

\begin{theorem}\label{th513} 
Given an integer \( m \geq 1 \)
and $0<\alpha<mn$.
Let $p,q, p_k, q_k$ be the same as in Proposition \ref{mixlorentz},
where $k=1,2,\ldots,m$. 
Suppose that $\tilde p, \tilde p_k,$ satisfy that
$$
\frac{1}{\tilde p}=\sum_{k=1}^m\frac{1}{\tilde p_k},\quad
\frac{1}{\tilde p}=\frac{1}{p_k}-\frac{\alpha}{mn}.
$$
If 
$\mathcal T$ is an multilinear fractional integral operator, 
then $\mathcal T$ can be extended to a bounded operator from
$H^{p_1,q_1}\times \cdots \times H^{p_m,q_m}$ into $L^{\tilde p, q}$.
\end{theorem}

\subsection{Hardy--Orlicz spaces}\label{s3s5}
A function $\Phi\colon[0,\infty)\to[0,\infty)$ is called an \emph{Orlicz function} if it is non-decreasing and satisfies $\Phi(0)=0$, $\Phi(t)>0$ whenever $t\in(0,\infty)$, and $\lim_{t\to\infty}\Phi(t)=\infty$.
Let $p\in[0,\infty)$ and $\varphi\colon\mathbb R^{n}\times[0,\infty)\to[0,\infty)$ be a function such that, for almost every $x\in\mathbb R^{n}$, $\varphi(x,\cdot)$ is an Orlicz function. The function $\varphi$ is said to be of \emph{uniformly upper} (resp., \emph{lower}) \emph{type} $p\in[0,\infty)$ if there exists a positive constant $C$ such that, for any $x\in\mathbb R^{n}$, $t\in[0,\infty)$ and $s\in[1,\infty)$ (resp.,$s\in[0,1]$), $\varphi(x,st)\leq Cs^{p}\varphi(x,t)$.
The function $\varphi$ is said to satisfy the \emph{uniform Muckenhoupt condition} for some $q\in[1,\infty)$, denoted by $\varphi\in \mathbb A_{q}(\mathbb R^{n})$, if, when $q\in(1,\infty)$,
\[
[\varphi]_{\mathbb A_{q}}:=\sup_{t\in(0,\infty)}\sup_{B\subset\mathbb R^{n}}\frac{1}{\vert B\vert^{q}}\int_{B}\varphi(x,t)dx\left\{ \int_{B}[\varphi(y,t)]^{-\frac{q^{\prime}}{q}}dy \right\}^{\frac{q}{q^{\prime}}}<\infty,
\]
where $1/q+1/q^{\prime}=1$, or
\[
[\varphi]_{\mathbb A_{1}}:=\sup_{t\in(0,\infty)}\sup_{B\subset\mathbb R^{n}}\frac{1}{\vert B\vert}\int_{B}\varphi(x,t)dx\left( \operatorname*{ess\,sup}\limits_{y\in\mathbb B}[\varphi(y,t)]^{-1} \right)<\infty.
\]
The class $\mathbb A_{\infty}(\mathbb R^{n})$ is defined by setting
\[
\mathbb A_{\infty}(\mathbb R^{n}):=\bigcup_{q\in[1,\infty)}\mathbb A_{q}(\mathbb R^{n}).
\]
For any given $\varphi\in\mathbb A_{\infty}(\mathbb R^{n})$, the \emph{critical weight index} $q(\varphi)$ is defined by setting
\[
q(\varphi):=\inf\{q\in[1,\infty)\colon\varphi\in\mathbb A_{q}(\mathbb R^{n})\}.
\]
Then the function $\varphi\colon\mathbb R^{n}\times[0,\infty)\to[0,\infty)$ is called a \emph{growth function} if the following hold true:
\begin{enumerate}
    \item $\varphi$ is a \emph{Musielak--Orlicz function}, namely,\\
    $\varphi(x,\cdot)$ is an Orlicz function for almost every given $x\in\mathbb R^{n}$;\\
    $\varphi(\cdot,t)$ is a measurable function for any given $t\in[0,\infty)$.
    \item $\varphi\in\mathbb A_{\infty}(\mathbb R^{n})$.
    \item The function $\varphi$ is of uniformly lower type $p$ for some $p\in(0,1]$ and of uniformly upper type 1.
\end{enumerate}
\begin{definition}
    For a growth function $\varphi$, a measurable function $f$ on $\mathbb R^{n}$ is said to be in the \emph{Musielak--Orlicz space} $L^{\varphi}(\mathbb R^{n})$ if $\int_{\mathbb R^{n}}\varphi(x,\vert f(x)\vert)dx<\infty$, equipped with the (quasi-)norm
    \[
    \|f\|_{L^{\varphi}}:=\inf\left\{  \lambda\in(0,\infty)\colon\int_{\mathbb R^{n}}\varphi\left( x,\frac{\vert f(x)\vert}{\lambda} \right)dx\leq1\right\}.
    \]
\end{definition}
    
It is shown in \cite[Subsection 7.7]{SHYY} that $L^{\varphi}(\mathbb R^{n})$ is a ball quasi-Banach function space.    
Let $X:=L^{\varphi}(\mathbb R^{n})$ with uniformly lower type $p_{\varphi}^{-}$ and uniformly upper type $p_{\varphi}^{+}$. From \cite[Theorem 7.14(i)]{SHYY}, Assumption~\ref{ass2.7} holds true when $s\in(0,1]$ and $\theta\in(0,\min\{s,p^{-}_{\varphi}/q(\varphi)\})$. Besides,
by \cite[Theorem 7.12]{SHYY} and \cite[Remark 2.7(g)]{WYY},
Assumption~\ref{ass2.8} holds true when $s\in(0,p_\varphi^-)$ and $q\in(1,\infty]$. 
Moreover, we also recall the following H\"{o}lder inequality on Orlicz spaces in \cite[Theorem 2.12]{Wa}.

\begin{proposition}\label{mixorlicz}
Let $L^{\varphi_k}(\mathbb R^{n})$ with uniformly lower type $p_{\varphi_k}^{-}$ and uniformly upper type $p_{\varphi_k}^{+}$,
where $k=1,\dots, m$. Define $\varphi=(\prod_{k=1}^m\varphi_k^{-1})^{-1}$. 
Then for all $f_k\in L^{\varphi_k}(\mathbb R^{n})$, we have
$$\left\|\prod_{k=1}^mf_k\right\|_{L^{\varphi}}\lesssim \prod_{k=1}^m\|f_k\|_{L^{\varphi_k}}.$$
\end{proposition}

Combining the above results with \cite[Lemma 2.10]{CH},
we deduce that Theorem~\ref{s1t1} also holds true when $Y= L^{\Psi}(\mathbb R^{n})$, 
$X_1={L^{\varphi_1}}(\mathbb R^{n})$, $\dots,$
$X_m={L^{\varphi_m}}(\mathbb R^{n})$, where $\Psi_k^{-1}:=t^{-\frac{\alpha}{mn}}\varphi_k^{-1}$ and $\Psi=(\prod_{k=1}^m\Psi_k^{-1})^{-1}$.
Denote $H^{\varphi}(\mathbb R^{n})$ the Hardy--Orlicz space as in Definition \ref{hx} with $X$ replaced by $L^{\varphi}(\mathbb R^{n})$. 
For more information on the Hardy--Orlicz space, see \cite{CFYY,LFFY,SHYY,WYY,YLK}.

\begin{theorem}\label{th516} 
Let $L^{\varphi}(\mathbb R^{n}),  L^{\varphi_k}(\mathbb R^{n})$ be the same as in Proposition \ref{mixorlicz},
where $k=1,2,\ldots,m$. Let $\Psi_k^{-1}:=t^{-\frac{\alpha}{mn}}\varphi_k^{-1}$ and $\Psi=(\prod_{k=1}^m\Psi_k^{-1})^{-1}$.
If 
$\mathcal T$ is an multilinear fractional integral operator, 
then $\mathcal T$ can be extended to a bounded operator from
$H^{\varphi_1}\times \cdots \times H^{\varphi_m}$ into $L^{\Psi}$.
\end{theorem}

\section{Proof of Theorem \ref{s1t1}}

In this section, we will prove the boundedness of
multilinear fractional integral operators on product of Hardy spaces
associated with ball quasi-Banach function spaces.
We now give the proof of our main result.\\

\noindent\textit{\bf Proof of Theorem \ref{s1t1}}\quad
Let $1\le k\le m$ and fix arbitrary functions $f_k\in H_{fin}^{X_k,\infty,N}(\mathbb R^{n})$.
By Theorem \ref{finite}, $f_k$ admits an atomic decomposition:
\begin{equation*}
  f_k=\sum_{j_k=1}^{M}\lambda_{k,j_k}a_{k,j_k},
\end{equation*}
where  $\lambda_{k,j_k}\ge 0$ and 
$a_{k,j_k}$ are $(X_k, \infty, N)$-atoms that fulfill
$$
{\rm supp}(a_{k,j_k})\subset Q_{k,j_k},\quad
\|a_{k,j_k}\|_{L^{\infty}}\leq\frac{1}{\|\chi_{Q_{k,j_k}}\|_{X_k}},\quad
\int_{\mathbb R^{n}}a_{k,j_k}(x)x^{\alpha}dx=0
$$ 
for any multi-index $\vert\alpha\vert\leq N$,
and 
\begin{equation*}
  \left\| \sum_{j_k=1}^{M} \frac{\lambda_{k,j_k}\chi_{Q_{k,j_k}}}{\|\chi_{Q_{k,j_k}}\|_{X_k}}   \right\|_{X_k}
  \lesssim \|f_k\|_{{H}_{X_k}}.
\end{equation*}

First we show that
\begin{align*}
\|\mathcal T_\alpha(f_1,\dots,f_m)\|_{Y}\lesssim \prod_{k=1}^m\left\| \sum_{j_k=1}^{M} \frac{\lambda_{k,j_k}\chi_{Q_{k,j_k}}}{\|\chi_{Q_{k,j_k}}\|_{X_k}}   \right\|_{X_k}.
\end{align*}

By using the multilinearity property of $\mathcal T_\alpha$, for a. e. $x\in\mathbb R^n$ we write
\begin{align*}
\mathcal T_\alpha(f_1,\dots,f_m)(x)=\sum_{j_1}\cdots\sum_{j_m}\lambda_{1,j_1}\cdots
\lambda_{m,j_m}\mathcal T_\alpha(a_{1,j_1},\dots,a_{m,j_m})(x).
\end{align*}

Fixed $j_1,\cdots,j_m$, we consider the two  cases for $x\in\mathbb R^n$
as follows.\\
Case 1: $x\in Q^\ast_{1,j_1}\cap\cdots\cap Q^\ast_{m,j_m}$;\\
Case 2: $x\in Q^{\ast,c}_{1,j_1}\cup\cdots\cup Q^{\ast,c}_{m,j_m}$.

Then we have
\begin{align*}\label{}
&|\mathcal T_\alpha(f_1,\dots,f_m)(x)|\\
&\le\sum_{j_1}\cdots\sum_{j_m}|\lambda_{1,j_1}|
\cdots|\lambda_{m,j_m}||\mathcal T_\alpha(a_{1,j_1},\ldots,a_{m,j_m})(x)|
\chi_{Q^\ast_{1,j_1}\cap\cdots\cap Q^\ast_{m,j_m}}(x)\\
&\quad+\sum_{j_1}\cdots\sum_{j_m}|\lambda_{1,j_1}|
\cdots|\lambda_{m,j_m}||\mathcal T_\alpha(a_{1,j_1},\ldots,a_{m,j_m})(x)|
\chi_{Q^{\ast,c}_{1,j_1}\cup\cdots\cup Q^{\ast,c}_{m,j_m}}(x)\\
&=:I_1(x)+I_2(x).
\end{align*}

Now we estimate the term $I_1(x)$.
To end it, we need to prove that
\begin{align*}
\begin{split}
\|I_1\|_{Y}
\lesssim \prod_{k=1}^m\left\| \sum_{j_k=1}^{M} \frac{\lambda_{k,j_k}\chi_{Q_{k,j_k}}}{\|\chi_{Q_{k,j_k}}\|_{X_k}}   \right\|_{X_k}.
\end{split}
\end{align*}
For fixed $j_1,\ldots,j_m$,
assume that $Q^\ast_{1,j_1}\cap\cdots\cap Q^\ast_{m,j_m}\neq 0$, otherwise
there is nothing to prove. Without loss of generality,
suppose that $Q_{1,j_1}$ has the smallest size among all these cubes.
We can pick a cube $G_{j_1,\cdots,j_m}$ such that
$$Q^\ast_{1,j_1}\cap\cdots\cap Q^\ast_{m,j_m}\subset
G_{j_1,\ldots,j_m}\subset G^\ast_{j_1,\ldots,j_m}\subset
Q^{\ast\ast}_{1,j_1}\cap\cdots\cap Q^{\ast\ast}_{m,j_m}$$
and $|G_{j_1,\ldots,j_m}|\ge C|Q_{1,j_1}|$.

We denote that $H(x):=|\mathcal T_\alpha(a_{1,j_1},\ldots,a_{m,j_m})(x)|
\chi_{Q_{1,j_1}^\ast\cap\cdots\cap Q_{m,j_m}^\ast}(x)$.
Hence , $$\mbox{supp}\,H(x)\subset
{Q_{1,j_1}^\ast\cap\cdots\cap Q_{m,j_m}^\ast}\subset
G_{j_1,\cdots,j_m}.$$

Now we can estimate $H$ as follows:
By using the fact that
$\mathcal T_\alpha$ maps $L^{r_1}\times L^\infty\times\cdots
\times L^\infty$ into $L^{s}$ when $r_1>1$ and
$\frac{1}{s}=\frac{1}{r_1}-\frac{\alpha}{n}$, we
get that
\begin{align*}
\begin{split}
\|H\|_{L^{s}}&\le
\|\mathcal T_\alpha(a_{1,j_1},\ldots,a_{m,j_m})\|_{L^{s}}\\
&\le
C\|a_{1,j_1}\|_{L^{r_1}}\|a_{2,j_2}\|_{L^\infty}
\cdots\|a_{m,j_m}\|_{L^\infty}\\
&\le C|Q_{1,j_1}|^{\frac{1}{r_1}}\|\chi_{Q_{1,j_1}}\|^{-1}_{X_1}
\|a_{2,j_2}\|_{L^\infty}\|a_{m,j_m}\|_{L^\infty}\\
&\le C
|Q_{1,j_1}|^{\frac{1}{r_1}}\prod_{k=1}^m\|\chi_{Q_{k,j_k}}\|^{-1}_{X_k}\\
&\le C
|G_{j_1,\ldots,j_m}|^{\frac{1}{s}}
|G_{j_1,\ldots,j_m}|^{\frac{\alpha}{n}}\prod_{k=1}^m\|\chi_{Q_{k,j_k}}\|^{-1}_{X_k}.
\end{split}
\end{align*}

Then it follows that
\begin{align*}
\begin{split}
\left(\frac{1}{|G_{j_1,\ldots,j_m}|}\int_{G_{j_1,\ldots,j_m}}H(x)^sdx\right)^{\frac{1}{s}}
\lesssim 
|G_{j_1,\ldots,j_m}|^{\frac{\alpha}{n}}\prod_{k=1}^m\|\chi_{Q_{k,j_k}}\|_{X_k}.
\end{split}
\end{align*}

By Lemma \ref{l2.7},
we conclude that
\begin{align*}
\begin{split}
\|I_1\|_{Y}&=\left\|\sum_{j_1}\cdots\sum_{j_m}|\lambda_{1,j_1}|
\cdots|\lambda_{m,j_m}||\mathcal T_\alpha(a_{1,j_1},\ldots,a_{m,j_m})|
\chi_{Q^\ast_{1,j_1}\cap\cdots\cap Q^\ast_{m,j_m}}\right\|_{Y}\\
&\lesssim
\left\|\sum_{j_1}\cdots\sum_{j_m}|\lambda_{1,j_1}|
\cdots|\lambda_{m,j_m}|\left(\frac{1}{|G_{j_1,\ldots,j_m}|}\int_{G_{j_1,\ldots,j_m}}H(y)^sdy\right)^{\frac{1}{s}}
\chi_{G_{j_1,\ldots,j_m}}\right\|_{Y}\\
&\lesssim
\left\|\sum_{j_1}\cdots\sum_{j_m}|\lambda_{1,j_1}|
\cdots|\lambda_{m,j_m}|
|G_{j_1,\ldots,j_m}|^{\frac{\alpha}{n}}\prod_{k=1}^m\|\chi_{Q_{k,j_k}}\|^{-1}_{X_k}
\chi_{G_{j_1,\ldots,j_m}}\right\|_{Y}\\
&\lesssim\left\|\prod_{k=1}^m
\sum_{j_k}
|\lambda_{k,j_k}|\|\chi_{Q_{k,j_k}}\|^{-1}_{X_k}
|Q_{k,j_k}|^{\frac{\alpha}{mn}}
\chi_{Q^{\ast\ast}_{k,j_k}}
\right\|_{Y}.
\end{split}
\end{align*}

Observe that $\frac{1}{q_k}=\frac{1}{p_k}-\frac{\alpha}{mn}$, where $k=1,\ldots,m$.
By applying (\ref{holder}) and Lemma \ref{l2.8},  we obtain that
\begin{align*}
\begin{split}
&\left\|\prod_{k=1}^m \sum_{j_k=1}^M |\lambda_{k,j_k}| |Q_{k,j_k}|^{\frac{\alpha}{mn}}
\chi_{Q^{\ast\ast}_{k,j_k}}
\right\|_{Y}
\lesssim \prod_{k=1}^m \left\|\sum_{j_k=1}^M |\lambda_{k,j_k}|\|\chi_{Q_{k,j_k}}\|^{-1}_{X_k} |Q_{k,j_k}|^{\frac{\alpha}{mn}}
\chi_{Q^{\ast\ast}_{k,j_k}}
\right\|_{Y_k}\\
&\lesssim \prod_{k=1}^m \left\|\sum_{j_k=1}^M |\lambda_{k,j_k}|\|\chi_{Q_{k,j_k}}\|^{-1}_{X_k}
\chi_{Q^\ast_{k,j_k}}
\right\|_{X_k}
\lesssim \prod_{k=1}^m\left\| \sum_{j_k=1}^{M} \frac{\lambda_{k,j_k}\chi_{Q_{k,j_k}}}{\|\chi_{Q_{k,j_k}}\|_{X_k}}   \right\|_{X_k}.
\end{split}
\end{align*}

Next, following the similar argument to that the proof of \cite[Theorem 1.2]{Tan2020}, we deal with the estimate of $I_2$.
Let A be a nonempty
subset of $\{1,\ldots,m\}$, and we denote the cardinality of
$A$ by $|A|$, then $1\le|A|\le m$. Let $A^c=\{1,\ldots,m\}\backslash
A.$ 
For fixed $A$, assume that $Q_{\tilde k,j_{\tilde k}}$
is the smallest cubes in the set of cubes $Q_{k,j_k},
k\in A$.
Let $z_{\tilde k,j_{\tilde k}}$ is the center of the cube
$Q_{\tilde k,j_{\tilde k}}$.
Denote $K_\alpha(x,y_1,\ldots,y_m)
=|(x-y_1,\ldots,x-y_m)|^{-mn+\alpha}$.
Observe that for all $|\beta|=d+1,\;
\beta=(\beta_1,\ldots,\beta_m)$
\begin{align*}
\begin{split}
|\partial_{y_1}^{\beta_1}\cdots\partial_{y_m}^{\beta_m}
K_\alpha(x,y_1,\ldots,y_m)|
\le C|(x-y_1,\ldots,x-y_m)|^{-mn+\alpha-|\beta|}.
\end{split}
\end{align*}

By using the Taylor expansion we conclude that
\begin{align*}
\begin{split}
  &\mathcal T_\alpha(a_{1,k_1},\ldots,a_{m,k_m})(x) \\
=& \int_{(\mathbb R^n)^{m-1}}\prod_{k\ne \tilde k}a_{k,j_k}(y_k)
\int_{\mathbb R^n}\sum_{|\gamma|=d+1}
(\partial_{y_{\tilde k}}^\gamma
K_\alpha)(x,y_1,\ldots,\xi,\ldots,y_m)\frac{(y_{\tilde k}
-z_{\tilde k,j_{\tilde k}})^\gamma}{\gamma!}
a_{\tilde k}(y_{\tilde k})d\vec{y}
\end{split}
\end{align*}
for some $\xi$ on the line segment joining
$y_{\tilde k}$ to $z_{\tilde k,j_{\tilde k}}$,
where $P_{z_{\tilde k,j_{\tilde k}}}^d(x,y_1,\ldots,y_m)$
is the Taylor polynomial of
$K_\alpha(x,y_1,\ldots,y_m)$.
Since $x\in (Q_{\tilde k,j_{\tilde k}}^\ast)^c$,
we get
$|x-\xi|\ge \frac{1}{2}|x-z_{\tilde k,j_{\tilde k}}|$.
Similarly,
$|x-y_k|\ge \frac{1}{2}|x-z_{k,j_{k}}|$
for $y_k\in Q_{k,j_k}$, $k\in A\backslash\{\tilde k\}$.
By using the estimate for the kernel $K_\alpha$, we conclude that
\begin{align*}
\begin{split}
& \mathcal T_\alpha(a_{1,k_1},\ldots,a_{m,k_m})(x)\\\lesssim&
\prod_{k\in A}\frac{|Q_{k,j_k}|^{1+\frac{d+1}{n|A|}}}
{(|x-z_{k,j_k}|+l(Q_{k,j_k}))^{n+\frac{d+1}{|A|}-\frac{\alpha}{m}}\|\chi_{Q_{k,j_k}}\|_{X_k}}
\prod_{k\in A^c}\frac{l(Q_{k,j_k})^{\frac{\alpha}{m}}\chi_{Q^\ast_{k,j_k}}(x)}
{\|\chi_{Q_{k,j_k}}\|_{X_k}}.
\end{split}
\end{align*}

Then we obtain that
\begin{align*}
\begin{split}
\|I_2\|_{Y}
\lesssim&
\Bigg\|\sum_{j_1}\cdots\sum_{j_m}\prod_{k=1}^m|\lambda_{k,{j_k}}|
\prod_{k\in A}\frac{|Q_{k,j_k}|^{1+\frac{d+1}{n|A|}}}
{(|x-z_{k,j_k}|+l(Q_{k,j_k}))^{n+\frac{d+1}{|A|}-\frac{\alpha}{m}}}
\\
&\quad\quad\quad\quad\times
\prod_{k\in A^c}\frac{l(Q_{k,j_k})^{\frac{\alpha}{m}}\chi_{Q^\ast_{k,j_k}}(x)}
{\|\chi_{Q_{k,j_k}}\|_{X_k}}\Bigg\|_{Y}\\
\lesssim&
\Bigg\|
\prod_{k\in A}
\sum_{j_k}|\lambda_{k,{j_k}}|
\frac{|Q_{k,j_k}|^{1+\frac{d+1}{n|A|}}}
{(|x-z_{k,j_k}|+l(Q_{k,j_k}))^{n+\frac{d+1}{|A|}-\frac{\alpha}{m}}}
\\
&\quad\quad\quad\quad\times
\prod_{k\in A^c}\sum_{j_k}|\lambda_{k,{j_k}}|\frac{l(Q_{k,j_k})^{\frac{\alpha}{m}}\chi_{Q^\ast_{k,j_k}}(x)}
{\|\chi_{Q_{k,j_k}}\|_{X_k}}
\Bigg\|_{Y}.
\end{split}
\end{align*}

For convenience, we denote that
$$
U_A=\sum_{j_k}|\lambda_{k,{j_k}}|
\frac{|Q_{k,j_k}|^{1+\frac{d+1}{n|A|}}}
{(|x-z_{k,j_k}|+l(Q_{k,j_k}))^{n+\frac{d+1}{|A|}-\frac{\alpha}{m}}}
$$
and
$$
U_{A^c}=\sum_{j_k}|\lambda_{k,{j_k}}|\frac{l(Q_{k,j_k})^{\frac{\alpha}{m}}\chi_{Q^\ast_{k,j_k}}(x)}
{\|\chi_{Q_{k,j_k}}\|_{X_k}}.
$$

Then applying (\ref{holder}) we obtain that
\begin{align*}
\begin{split}
\|I_2\|_{Y}
&\lesssim
\left(\prod_{k\in A}\|
U_{A}\chi_{E_A}\|_{Y_k}\right)
\left(\prod_{k\in A^c}\|U_{A^c}
\chi_{E_A}\|_{Y_k}\right).
\end{split}
\end{align*}

Choose $d>1$ enough large such that $\theta:=\frac{n}{n+\frac{d+1}{|A|}-\frac{\alpha}{m}}\in (0,1]$.
By Assumption \ref{ass2.7} and Lemma \ref{l2.8}, we have
\begin{align*}
\begin{split}
\prod_{k\in A}\|U_{A}
\chi_{E_A}\|_{Y_k}
&=\prod_{k\in A}\left\|\sum_{j_k}|\lambda_{k,j_k}|
\frac{l(Q_{k,j_k})^{n+\frac{d+1}{|A|}}}
{(|x-z_{k,j_k}|+l(Q_{k,j_k}))^{n+\frac{d+1}
{|A|}-\frac{\alpha}{m}}}\right\|_{Y_k}\\
&\lesssim
\prod_{k\in A}
\left\|\sum_{j_k}|\lambda_{k,j_k}|l(Q_{k,j_k})^{\frac{\alpha}{m}}{M^{(\theta)}({\chi_{Q_{j,k_j}}})}
\right\|_{Y_k}\\
&\lesssim
\prod_{k\in A}\left\|\sum_{j_k}
|\lambda_{k,j_k}||Q_{k,j_k}|^{\frac{\alpha}{mn}}\chi_{Q_{k,j_k}}
\right\|_{Y_k}\\
&\lesssim \prod_{k=1}^m\left\| \sum_{j_k} \frac{\lambda_{k,j_k}\chi_{Q_{k,j_k}}}{\|\chi_{Q_{k,j_k}}\|_{X_k}}   \right\|_{X_k}.
\end{split}
\end{align*}

Applying Lemma \ref{l2.8} and Assumption \ref{ass2.7} again yield that

\begin{align*}
\begin{split}
\prod_{k\in A^c}\|U_{A^c}
\chi_{E_A}\|_{Y_k}
&=\prod_{k=A^c}\left\|\sum_{j_k}|\lambda_{k,{j_k}}|\frac{l(Q_{k,j_k})^{\frac{\alpha}{m}}\chi_{Q^\ast_{k,j_k}}(x)}
{\|\chi_{Q_{k,j_k}}\|_{X_k}}\right\|_{Y_k}\\
&\lesssim \prod_{k=1}^m\left\| \sum_{j_k} \frac{\lambda_{k,j_k}\chi_{Q_{k,j_k}}}{\|\chi_{Q_{k,j_k}}\|_{X_k}}   \right\|_{X_k}.
\end{split}
\end{align*}
Therefore, combining the estimates for $I_1$ and $I_2$ we get that
\begin{align*}
\|\mathcal T_\alpha(f_1,\dots,f_m)\|_{Y}\lesssim \prod_{k=1}^m\left\| f \right\|_{X_k},
\end{align*}
which proves Theorem \ref{s1t1}.
$\hfill\Box$

{\bf Acknowledgments.}
The project is sponsored by
the National Natural Science Foundation of China
(Grant No. 11901309), the Natural Science Foundation of Nanjing University of Posts and Telecommunications (Grant No. NY224167)
and the Jiangsu Government Scholarship for Overseas Studies.



\begin{thebibliography}{20}

\bibitem{BP} A. Benedek and R. Panzone, {\it The space $L^{p}$, with mixed norm}, Duke Math. J. {\bf 28} (1961), no. 3, 301--324.

\bibitem{CRS} M.~J.~Carro, J.~A.~Raposo and J.~Soria, \emph{{Recent developments in the theory of Lorentz spaces and weighted inequalities}},  Mem. Amer. Math. Soc. \textbf{187} (2007), no.~877,  xii+128 pp.

\bibitem{CFYY} D. Chang, Z. Fu, D. Yang and S. Yang, {\it Real-variable characterizations of Musielak--Orlicz--Hardy spaces associated with Schr\"odinger operators on domains,}  Math. Methods Appl. Sci., {\bf 39}(2016), no. 3,:  533--569.

\bibitem{CT} X. Chen and J. Tan, {\it Local Hardy spaces associated with ball quasi-Banach function spaces and their dual spaces}, arXiv:2406.10841

\bibitem{CJY} Y. Chen, H. Jia and D. Yang, {\it Boundedness of fractional integrals on Hardy spaces associated with ball quasi-Banach function spaces}
Tokyo J. Math. {\bf 47} (2024), no. 1, 19--59.



\bibitem{CF} D. Cruz-Uribe and A. Fiorenza, \textit{Variable Lebesgue spaces:
Foundations and Harmonic Analysis}, Birkh\"{a}user, Basel, 2013.

\bibitem{CFMP} D. Cruz-Uribe, A. Fiorenza, J. Martell and C. P\'{e}rez,
\textit{The boundedness of classical operators on variable $L^p$
spaces}, Ann. Acad. Sci. Fenn. Math. {\bf 31} (2006), 239-264.

\bibitem{CH} D. Cruz-Uribe and P. H\"ast\"o,
{\it Extrapolation and interpolation in generalized Orlicz spaces}, Trans. Am. Math. Soc. {\bf 370} (6) (2018), 4323--4349.

\bibitem{CMN} D. Cruz-Uribe, K. Moen, and H. V. Nguyen. {\it The boundedness of multilinear Calder\'onZygmund operators on weighted and variable Hardy spaces,} Publ. Mat., {\bf 63} (2019), no. 2, 679-713.

\bibitem{CMN1} D. Cruz-Uribe, K. Moen, and H. V. Nguyen. {\it Multilinear
fractional Calder\'on--Zygmund operators on weighted Hardy spaces,}
Houston J. Math. {\bf45} (2019), no. 3, 853-871.

\bibitem{CW} D. Cruz-Uribe and L. Wang,
\textit{Variable Hardy spaces}, {Indiana Univ. Math. J.}, \textbf{63} (2014), 447-493.


\bibitem{FS} C. Fefferman and E. Stein,
\textit{$H^p$ spaces of several variables},
Acta Math. {\bf129}, (1972), no. 3-4, 137-193.


\bibitem{G} L. Grafakos,
\textit{On multilinear fractional integrals},
Studia Math. \textbf{102} (1992), 49-56.

\bibitem{GK} L. Grafakos and N. Kalton,
\textit{Multilinear Calder\'on-Zygmund operators on Hardy spaces,}
Collect. Math. \textbf{52} (2001), no. 2, 169-179.


\bibitem{HN}  E. Hale and V. Naibo, {\it Fractional Leibniz rules in the setting of quasi-Banach function spaces}, J. Fourier Anal. Appl.  {\bf 29} (2023), no. 5, Paper No. 64, 1--46.


\bibitem{HL} J. Huang and Y. Liu,
\textit{The boundedness of multilinear Calder\'on-Zygmund operators
on Hardy spaces,} Proc. Indian Acad. Sci. Math. Sci.
{\bf123} (2013), no. 3, 383-392.

\bibitem{HLYY}
L.~Huang, J.~Liu, D.~Yang, and W.~Yuan, {\it Atomic and Littlewood--Paley characterizations of anisotropic mixed-norm Hardy spaces and their applications}, J. Geom. Anal. \textbf{29} (2019), 1991--2067.

\bibitem{HY}
L.~Huang and D.~Yang, {\it On function spaces with mixed-norms--a survey}, J. Math. Study \textbf{54} (2021), 262--336.

\bibitem{KS} C. E. Kenig, E. M. Stein,
\textit{Multilinear estimates and fractional integration,}
Math. Res. Lett. \textbf{6} (1999), 1-15.

\bibitem{KR} O. Kov\'{a}\v{c}ik and J. R\'{a}kosn\'{\i}k,
\textit{On spaces $L^{p(x)}$ and $W^{k,p(x)}$,}
{Czechoslovak Math. J.} \textbf{41}
(1991), 592-618.


\bibitem{LFFY} B. Li, X. Fan, Z. Fu and D. Yang, {\it Molecular characterization of anisotropic Musielak--Orlicz Hardy spaces and their applications}, Acta Math. Sin. (Engl. Ser.) {\bf32} (2016), no. 11, 1391--1414.

\bibitem{LXY}
W. Li, Q. Xue and K. Yabuta,
{\it Maximal operator for multilinear Calder\'on-Zygmund
Singular integral operators on weighted Hardy spaces},
J. Math. Anal. Appl. {\bf373} (2011), 384-392.

\bibitem{LL}
Y. Lin, S. Lu, \textit{Boundedness of multilinear singular integral
operators on Hardy and Herz-Hardy spaces,}
Hokkaido Math. J. {\bf 36} (2007), no. 3, 585-613.

\bibitem{Lo}
G.~G.~Lorentz, \emph{{On the theory of spaces $\Lambda$}}, Pacific J. Math. \textbf{1} (1951), 411--429.

\bibitem{NS} E. Nakai and Y. Sawano, \textit{Hardy spaces with variable exponents and generalized Campanato spaces,}
{J. Funct. Anal.}, \textbf{262} (2012), 3665-3748.





\bibitem{SHYY}
Y.~Sawano, K.-P.~Ho, D.~Yang, and S.~Yang, \emph{{Hardy spaces for ball quasi-Banach function spaces}}, Dissertationes Math. \textbf{525} (2017), 1--102.

\bibitem{SW}
E. M. Stein and G. Weiss,
{\it On the theory of harmonic functions of several variables. I. The theory of $H^p$-spaces,}
Acta Math. {\bf103} (1960), 25--62.

\bibitem{Tan} J. Tan, {\it Boundedness of maximal operator
for multilinear Calder\'on--Zygmund operators on
products of variable Hardy spaces},
Kyoto J. Math., 60 (2020), no. 2, 561-574.

\bibitem{Tan1} J. Tan, {\it Bilinear Calder\'on--Zygmund operators on products of variable Hardy spaces}, Forum Math. 31 (2019), no. 1, 187-198.

\bibitem{Tan2020} J. Tan, {\it Boundedness of multilinear fractional type operators on Hardy spaces with variable exponents},
Anal. Math. Phys. {\bf10} (2020), no. 4, Paper No. 70, 16 pp.

\bibitem{Tan24} J. Tan, {\it 
Product Hardy spaces meet ball quasi-Banach function spaces}
J. Geom. Anal. {\bf 34} (2024), no. 3, Paper No. 92, 33 pp.

\bibitem{TLZ} J. Tan, Z. Liu and J. Zhao,
\textit{On multilinear commutators in
variable Lebesgue spaces},
J. Math. Inequal. , {\bf 11} (2017)
no. 3, 715-734.

\bibitem{TZ1} J. Tan and J. Zhao, \textit{Multilinear pseudo-differential operators on product of local Hardy spaces with variable exponents}, J. Pseudo-Differ. Oper. Appl., 10 (2019), no. 2, 379-396.

\bibitem{WYY}
F.~Wang, D.~Yang, and S.~Yang, \emph{{Applications of Hardy spaces associated
  with ball quasi-Banach function spaces}}, Results Math. \textbf{75} (2020),
no. 1, Paper No. 26, 1--58.

\bibitem{Wa}
S.~Wang, {\it H\"older's inequalities and multilinear singular integrals on generalized Orlicz spaces},
J. Math. Inequal. {\bf18} (2024), no. 3, 811--828.

\bibitem{YJY}
X. Yan, H. Jia and D. Yang, {\it A Fourier multiplier theorem on anisotropic Hardy spaces associated with ball quasi-Banach function spaces} Ann. Funct. Anal. {\bf16} (2025), no. 1, Paper No. 5. 1--37.

\bibitem{YYY}
X.~Yan, D.~Yang, and W.~Yuan, \emph{{Intrinsic square function
  characterizations of Hardy spaces associated with ball quasi-Banach function
  spaces}}, Front. Math. China \textbf{15} (2020), 769--806.

\bibitem{YLK}
D. Yang, Y. Liang, and L.D. Ky, {\it Real-variable theory of Musielak--Orlicz Hardy spaces,} In: Lecture Notes in Mathematics, vol. 2182. Springer, Cham,  2017.


\end{thebibliography}
\end{document}